\documentclass[12pt,reqno,a4paper]{amsart}%
\usepackage{amsfonts}
\usepackage{amsmath}
\usepackage{amssymb}
\usepackage{graphicx}%
\newtheorem{theorem}{Theorem}[section]
\newtheorem{corollary}[theorem]{Corollary}

\theoremstyle{definition}
\newtheorem{definition}[theorem]{Definition}
\newtheorem{remark}[theorem]{Remark}
\numberwithin{equation}{section} 
\begin{document}
\title[$\left(  p,q\right)  -$LUCAS POLYNOMIALS ... ]{FEKETE-SZEG\"{O} INEQUALITY FOR ANALYTIC AND BI-UNIVALENT FUNCTIONS
SUBORDINATE TO $\left(  p,q\right)  -$LUCAS POLYNOMIALS}
\author{ALA AMOURAH}
\address{ALA AMOURAH: Department of Mathematics, Faculty of Science and Technology,
Irbid National University, Irbid, Jordan.}
\email{alaammour@yahoo.com.}

\begin{abstract}
In the present paper, a subclass of analytic and bi-univalent functions by
means of $\left(  p,q\right)  -$ Lucas polynomials is introduced. Certain
coefficients bounds for functions belonging to this subclass are obtained.
Furthermore, the Fekete-Szeg\"{o} problem for this subclass is solved.

\end{abstract}
\maketitle

\section{Introduction}

Let $\mathcal{A}$ denote the class of all analytic functions $f$ defined in
the open unit disk $\mathbb{U}=\{z\in\mathbb{C}:\left\vert z\right\vert <1\}$
and normalized by the conditions $f(0)=0$ and $f^{\prime}(0)=1$. Thus each
$f\in\mathcal{A}$ has a Taylor-Maclaurin series expansion of the form:%

\begin{equation}
f(z)=z+\sum\limits_{n=2}^{\infty}a_{n}z^{n},\ \ (z\in\mathbb{U}).\text{
\ \ \ \ } \label{ieq1}%
\end{equation}

Further, let $\mathcal{S}$ denote the class of all functions $f\in\mathcal{A}$
which are univalent in $\mathbb{U}$ (for details, see \cite{Duren}; see also
some of the recent investigations \cite{C1,C4,C5,C2}).

Two of the important and well-investigated subclasses of the analytic and
univalent function class $\mathcal{S}$ are the class $\mathcal{S}^{\ast
}(\alpha)$ of starlike functions of order $\alpha$ in $\mathbb{U}$ and the
class $\mathcal{K}(\alpha)$ of convex functions of order $\alpha$ in
$\mathbb{U}$. By definition, we have%

\begin{equation}
\mathcal{S}^{\ast}(\alpha):=\left\{  f:\ f\in\mathcal{S}\ \ \text{and}%
\ \ \mbox{Re}\left\{  \frac{zf^{\prime}(z)}{f(z)}\right\}  >\alpha,\quad
(z\in\mathbb{U};0\leq\alpha<1)\right\}  , \label{d1}%
\end{equation}
and%
\begin{equation}
\mathcal{K}(\alpha):=\left\{  f:\ f\in\mathcal{S}\ \ \text{and}%
\ \ \mbox{Re}\left\{  1+\frac{zf^{\prime\prime}(z)}{f^{\prime}(z)}\right\}
>\alpha,\quad(z\in\mathbb{U};0\leq\alpha<1)\right\}  . \label{d2}%
\end{equation}

It is clear from the definitions (\ref{d1}) and (\ref{d2}) that $\mathcal{K}%
(\alpha)\subset\mathcal{S}^{\ast}(\alpha)$. Also we have%

\[
f(z)\in\mathcal{K}(\alpha)\ \ \text{iff}\ \ zf^{\prime}(z)\in\mathcal{S}%
^{\ast}(\alpha),
\]
and%
\[
f(z)\in\mathcal{S}^{\ast}(\alpha)\ \ \text{iff}\ \ \int_{0}^{z}\frac{f(t)}%
{t}dt=F(z)\in\mathcal{K}(\alpha).
\]

It is well-known that, if $f(z)$ is an univalent analytic function from a
domain $\mathbb{D}_{1}$ onto a domain $\mathbb{D}_{2}$, then the inverse
function $g(z)$ defined by%
\[
g\left(  f(z)\right)  =z,\ \ (z\in\mathbb{D}_{1}),
\]
is an analytic and univalent mapping from $\mathbb{D}_{2}$ to $\mathbb{D}_{1}%
$. Moreover, by the familiar Koebe one-quarter theorem (for details, see
\cite{Duren}), we know that the image of $\mathbb{U}$ under every function
$f\in\mathcal{S}$ contains a disk of radius $\frac{1}{4}$.

According to this, every function $f\in\mathcal{S}$ has an inverse map
$f^{-1}$ that satisfies the following conditions:

\begin{center}
$f^{-1}(f(z))=z \ \ \ (z\in\mathbb{U}),$
\end{center}

and

\begin{center}
$f\left(  f^{-1}(w)\right)  =w$ $\ \ \ \left(  |w|<r_{0}(f);r_{0}(f)\geq
\frac{1}{4}\right)  $.
\end{center}

In fact, the inverse function is given by%
\begin{equation}
f^{-1}(w)=w-a_{2}w^{2}+(2a_{2}^{2}-a_{3})w^{3}-(5a_{2}^{3}-5a_{2}a_{3}%
+a_{4})w^{4}+\cdots. \label{ieq2}%
\end{equation}

A function $f\in\mathcal{A}$ is said to be bi-univalent in $\mathbb{U}$ if
both $f(z)$ and $f^{-1}(z)$ are univalent in $\mathbb{U}$. Let $\Sigma$ denote
the class of bi-univalent functions in $\mathbb{U}$ given by (\ref{ieq1}).
Examples of functions in the class $\Sigma$ are%
\[
\frac{z}{1-z},\ -\log(1-z),\ \frac{1}{2}\log\left(  \frac{1+z}{1-z}\right)
,\cdots.
\]

It is worth noting that the familiar Koebe function is not a member of
$\Sigma$, since it maps the unit disk $\mathbb{U}$ univalently onto the entire
complex plane except the part of the negative real axis from $-1/4$ to
$-\infty$. Thus, clearly, the image of the domain does not contain the unit
disk $\mathbb{U}$. For a brief history and some intriguing examples of
functions and characterization of the class $\Sigma$, see Srivastava et al.
\cite{C27} and Yousef et al. \cite{C3, C33, C333}.

In 1967, Lewin \cite{C21} investigated the bi-univalent function class
$\Sigma$ and showed that $|a_{2}|<1.51$. Subsequently, Brannan and Clunie
\cite{C22} conjectured that $|a_{2}|\leq\sqrt{2}.$ On the other hand,
Netanyahu \cite{C23} showed that $\underset{f\in\Sigma}{\max}$ $|a_{2}%
|=\frac{4}{3}.$ The best known estimate for functions in $\Sigma$ has been
obtained in 1984 by Tan \cite{C24}, that is, $|a_{2}|<1.485$. The coefficient
estimate problem for each of the following Taylor-Maclaurin coefficients
$|a_{n}|$ $(n\in\mathbb{N}\backslash\{1,2\})$ for each $f\in\Sigma$ given by
(\ref{ieq1}) is presumably still an open problem.

\bigskip For the polynomials $p(x)$ and $q(x)$ with real coefficients, the
$\left(  p,q\right)  -$Lucas polynomials $L_{p,q,k}(x)$ are defined by the
following recurrence relation (see \cite{bul}):%

\[
L_{p,q,k}(x)=p(x)L_{p,q,k-1}(x)+q(x)L_{p,q,k-2}(x),\ \ (k\geq2),
\]
with%
\begin{equation}
L_{p,q,0}(x)=2,\text{ }L_{p,q,1}(x)=p(x)\text{, and }L_{p,q,2}(x)=p^{2}%
(x)+2q(x). \label{w1}%
\end{equation}

The generating function of the $\left(  p,q\right)  -$Lucas Polynomials
$L_{p,q,k}(x)$ (see \cite{8}) is given by%
\[
A_{\left\{  L_{p,q,k}(x)\right\}  }(z)=\sum\limits_{k=2}^{\infty}%
L_{p,q,k}(x)z^{k}=\frac{2-p(x)z}{1-p(x)z-q(x)z^{2}}.
\]

The concept of $\left(  p,q\right)  -$Lucas polynomials was introduced by
Alt\i nkaya and Yal\c{c}\i n \cite{9}.

\section{\bigskip The class $\mathfrak{B}_{\Sigma}^{\mu}(\alpha,\lambda
,\delta)$}

Firstly, we consider a comprehensive class of analytic bi-univalent functions
introduced and studied by Yousef et al. \cite{class} defined as follows:

\begin{definition}
\label{def22} (See \cite{class}) For $\lambda\geq1,$ $\mu\geq0,$ $\delta\geq0$
and $0\leq\alpha<1$, a function $f\in\Sigma$ given by (\ref{ieq1}) is said to
be in the class $\mathfrak{B}_{\Sigma}^{\mu}(\alpha,\lambda,\delta)$ if the
following conditions hold for all $z,w\in\mathbb{U}$:
\begin{equation}
\mbox{Re}\left(  (1-\lambda)\left(  \frac{f(z)}{z}\right)  ^{\mu}+\lambda
f^{\prime}(z)\left(  \frac{f(z)}{z}\right)  ^{\mu-1}+\xi\delta zf^{\prime
\prime}(z)\right)  >\alpha\label{ieq23}%
\end{equation}
and%
\begin{equation}
\mbox{Re}\left(  (1-\lambda)\left(  \frac{g(w)}{w}\right)  ^{\mu}+\lambda
g^{\prime}(w)\left(  \frac{g(w)}{w}\right)  ^{\mu-1}+\xi\delta wg^{\prime
\prime}(w)\right)  >\alpha, \label{ieq24}%
\end{equation}
where the function $g(w)=f^{-1}(w)$ is defined by (\ref{ieq2}) and $\xi
=\frac{2\lambda+\mu}{2\lambda+1}$.
\end{definition}

\begin{remark}
\label{rem2} In the following special cases of Definition \ref{def22}; we show
how the class of analytic bi-univalent functions $\mathfrak{B}_{\Sigma}^{\mu
}(\alpha,\lambda,\delta)$ for suitable choices of $\lambda$, $\mu$ and
$\delta$ lead to certain new as well as known classes of analytic bi-univalent
functions studied earlier in the literature.

(i) For $\delta=0,$ we obtain the bi-univalent function class $\mathfrak{B}%
_{\Sigma}^{\mu}(\alpha,\lambda,0):=\mathfrak{B}_{\Sigma}^{\mu}(\alpha
,\lambda)$ introduced by \c{C}a\u{g}lar et al. \cite{C29}. \vspace{0.05in}

(iii) For $\delta=0$, $\mu=1$, and $\lambda=1,$ we obtain the bi-univalent
function class $\mathfrak{B}_{\Sigma}^{1}(\alpha,1,0):=\mathfrak{B}_{\Sigma
}(\alpha)$ introduced by Srivastava et al. \cite{C27}. \vspace{0.05in}

(iv) For $\delta=0$, $\mu=0$, and $\lambda=1,$ we obtain the well-known class
$\mathfrak{B}_{\Sigma}^{0}(\alpha,1,0):=\mathcal{S}^{*}_{\Sigma}(\alpha)$ of
bi-starlike functions of order $\alpha$. \vspace{0.05in}

(iv) For $\mu=1$, we obtain the well-known class $\mathfrak{B}_{\Sigma}%
^{1}(\alpha,\lambda,\delta):=\mathfrak{B}_{\Sigma}(\alpha,\lambda,\delta)$ of
bi-univalent functions.
\end{remark}


\section{\bigskip Main Results}

We begin this section by defining the class $\mathfrak{B}_{\Sigma}^{\mu
}(\lambda,\delta)$ as follows:

\begin{definition}
\label{def221} For $\lambda\geq1,$ $\mu\geq0$ and $\delta\geq0$, a function
$f\in\Sigma$ given by (\ref{ieq1}) is said to be in the class $\mathfrak{B}%
_{\Sigma}^{\mu}(\lambda,\delta)$ if the following subordinations are
satisfied:%
\[
(1-\lambda)\left(  \frac{f(z)}{z}\right)  ^{\mu}+\lambda f^{\prime}(z)\left(
\frac{f(z)}{z}\right)  ^{\mu-1}+\xi\delta zf^{\prime\prime}(z)\prec
A_{\left\{  L_{p,q,k}(x)\right\}  }(z)-1
\]
and%
\[
(1-\lambda)\left(  \frac{f^{-1}(w)}{w}\right)  ^{\mu}+\lambda\left(
f^{-1}(w)\right)  ^{\prime}\left(  \frac{f^{-1}(w)}{w}\right)  ^{\mu-1}%
+\xi\delta z\left(  f^{-1}(w)\right)  ^{\prime\prime}\prec A_{\left\{
L_{p,q,k}(x)\right\}  }(w)-1,
\]
where $f^{-1}$is given by (\ref{ieq2}).
\end{definition}

\begin{theorem}
\bigskip\label{thm111} For $\lambda\geq1,$ $\mu\geq0$ and $\delta\geq0$, let
$f\in\mathcal{A}$ belongs to the class $\mathfrak{B}_{\Sigma}^{\mu}%
(\lambda,\delta).$ Then%
\[
\left\vert a_{2}\right\vert \leq\frac{2\left\vert p(x)\right\vert
\sqrt{\left\vert p(x)\right\vert }}{\sqrt{\left\vert  \left(
\mu+2\lambda\right)  \left[  1+\mu+\frac{12\delta}{2\lambda+1}\right]
p^{2}(x)-2\left(  \mu+\lambda+2\xi\delta\right)  ^{2}\left(  p^{2}%
(x)+2q(x)\right)  \right\vert }},
\]
and%
\[
\left\vert a_{3}\right\vert \leq\frac{p^{2}(x)}{\left(  \mu+\lambda+2\xi
\delta\right)  ^{2}}+\frac{\left\vert p(x)\right\vert }{\left(  \mu
+2\lambda+2\xi\delta\right)  }.
\]

\end{theorem}

\begin{proof}
Let $f\in\mathfrak{B}_{\Sigma}^{\mu}(\lambda,\delta).$ From Definition
\ref{def221}, for some analytic functions $\phi,$ $\psi$ such that
$\phi\left(  0\right)  =\psi\left(  0\right)  =0$ and $\left\vert \phi\left(
z\right)  \right\vert <1,$ $\left\vert \psi\left(  w\right)  \right\vert <1$
for all $z,w\in\mathbb{U},$ then we can write%
\begin{align}
& \hspace{-.7in} (1-\lambda)\left(  \frac{f(z)}{z}\right)  ^{\mu}+\lambda f^{\prime
}(z)\left(  \frac{f(z)}{z}\right)  ^{\mu-1}+\xi\delta zf^{\prime\prime
}(z)\label{p1}\\
& \hspace{.7in}  =-1+L_{p,q,0}(x)+L_{p,q,1}(x)\phi\left(  z\right)  +L_{p,q,2}(x)\phi
^{2}\left(  z\right)  +\cdots\nonumber
\end{align}
and%
\begin{align}
& \hspace{-.6in} (1-\lambda)\left(  \frac{f^{-1}(w)}{w}\right)  ^{\mu}+\lambda\left(
f^{-1}(w)\right)  ^{\prime}\left(  \frac{f^{-1}(w)}{w}\right)  ^{\mu-1}%
+\xi\delta z\left(  f^{-1}(w)\right)  ^{\prime\prime}\label{p2}\\
& \hspace{.9in} =-1+L_{p,q,0}(x)+L_{p,q,1}(x)\psi\left(  w\right)  +L_{p,q,2}(x)\psi
^{2}\left(  w\right)  +\cdots.\nonumber
\end{align}

From the equalities (\ref{p1}) and (\ref{p2}), we obtain that%
\begin{align}
& \hspace{-.7in} (1-\lambda)\left(  \frac{f(z)}{z}\right)  ^{\mu}+\lambda f^{\prime
}(z)\left(  \frac{f(z)}{z}\right)  ^{\mu-1}+\xi\delta zf^{\prime\prime
}(z)\label{p3}\\
& \hspace{.7in} =1+L_{p,q,1}(x)r_{1}z+\left[  L_{p,q,1}(x)r_{2}+L_{p,q,2}(x)r_{1}%
^{2}\right]  z^{2}+\cdots\nonumber
\end{align}
and%
\begin{align}
& \hspace{-.6in} (1-\lambda)\left(  \frac{f^{-1}(w)}{w}\right)  ^{\mu}+\lambda\left(
f^{-1}(w)\right)  ^{\prime}\left(  \frac{f^{-1}(w)}{w}\right)  ^{\mu-1}%
+\xi\delta z\left(  f^{-1}(w)\right)  ^{\prime\prime}\label{p4}\\
& \hspace{.9in} =1+L_{p,q,1}(x)s_{1}w+\left[  L_{p,q,1}(x)s_{2}+L_{p,q,2}(x)s_{1}%
^{2}\right]  w^{2}+\cdots.\nonumber
\end{align}

It is fairly well known that if%
\[
\left\vert \phi\left(  z\right)  \right\vert =\left\vert r_{1}z+r_{2}%
z^{2}+r_{3}z^{3}+\cdots\right\vert <1,\ \ (z\in\mathbb{U})
\]
and%
\[
\left\vert \psi\left(  w\right)  \right\vert =\left\vert s_{1}w+s_{2}%
w^{2}+s_{3}w^{3}+\cdots\right\vert <1,\ \ (w\in\mathbb{U}),
\]
then
\begin{equation}
\left\vert r_{k}\right\vert <1 \text{ and } \left\vert s_{k}\right\vert <1\text{ for }k\in%
\mathbb{N}
. \label{w2}%
\end{equation}

Thus, upon comparing the corresponding coefficients in (\ref{p3}) and
(\ref{p4}), we have%
\begin{equation}
\left(  \mu+\lambda+2\xi\delta\right)  a_{2}=L_{p,q,1}(x)r_{1}, \label{p5}%
\end{equation}

\begin{equation}
\left(  \mu+2\lambda\right)  \left[  \left(  \frac{\mu-1}{2}\right)  a_{2}%
^{2}+\left(  1+\frac{6\delta}{2\lambda+1}\right)  a_{3}\right]  =L_{p,q,1}%
(x)r_{2}+L_{p,q,2}(x)r_{1}^{2}, \label{p6}%
\end{equation}

\begin{equation}
-\left(  \mu+\lambda+2\xi\delta\right)  a_{2}=L_{p,q,1}(x)s_{1}, \label{p7}%
\end{equation}
and%
\begin{equation}
\left(  \mu+2\lambda\right)  \left[  \left(  \frac{\mu+3}{2}+\frac{12\delta
}{2\lambda+1}\right)  a_{2}^{2}-\left(  1+\frac{6\delta}{2\lambda+1}\right)
a_{3}\right]  =L_{p,q,1}(x)s_{2}+L_{p,q,2}(x)s_{1}^{2}. \label{p8}%
\end{equation}

It follows from (\ref{p5}) and (\ref{p7}) that%
\begin{equation}
r_{1}=-s_{1} \label{p9}%
\end{equation}
and%
\begin{equation}
2\left(  \mu+\lambda+2\xi\delta\right)  ^{2}a_{2}^{2}=L_{p,q,1}^{2}(x)\left(
r_{1}^{2}+s_{1}^{2}\right)  . \label{p10}%
\end{equation}

If we add (\ref{p6}) and (\ref{p8}), we get%
\begin{equation}
\left(  \mu+2\lambda\right)  \left[  1+\mu+\frac{12\delta}{2\lambda+1}\right]
a_{2}^{2}=L_{p,q,1}(x)\left(  r_{2}+s_{2}\right)  +L_{p,q,2}(x)\left(
r_{1}^{2}+s_{1}^{2}\right)  . \label{p111}%
\end{equation}

Substituting the value of $\left(  r_{1}^{2}+s_{1}^{2}\right)  $ from
(\ref{p10}) in the right hand side of (\ref{p111}), we deduce that%
\begin{align}
&  \left[  \left(  \mu+2\lambda\right)  \left[  1+\mu+\frac{12\delta}%
{2\lambda+1}\right]  L_{p,q,1}^{2}(x)-2\left(  \mu+\lambda+2\xi\delta\right)
^{2}L_{p,q,2}(x)\right]  a_{2}^{2}\label{p11} \nonumber\\
& \hspace{1.2in} =L_{p,q,1}^{3}(x)\left(  r_{2}+s_{2}\right)  .
\end{align}

Moreover computations using (\ref{w1}), (\ref{w2}) and (\ref{p11}), we find
that%
\[
\left\vert a_{2}\right\vert \leq\frac{2\left\vert p(x)\right\vert
\sqrt{\left\vert p(x)\right\vert }}{\sqrt{\left\vert \left(
\mu+2\lambda\right)  \left[  1+\mu+\frac{12\delta}{2\lambda+1}\right]
p^{2}(x)-2\left(  \mu+\lambda+2\xi\delta\right)  ^{2}\left(  p^{2}%
(x)+2q(x)\right)  \right\vert }}.
\]

Moreover, if we subtract (\ref{p8}) from (\ref{p6}), we obtain%
\begin{equation}
2\left(  \mu+2\lambda\right)  \left(  1+\frac{6\delta}{2\lambda+1}\right)
\left(  a_{3}-a_{2}^{2}\right)  =L_{p,q,1}(x)\left(  r_{2}-s_{2}\right)
+L_{p,q,2}(x)\left(  r_{1}^{2}-s_{1}^{2}\right)  . \label{p12}%
\end{equation}

Then, in view of (\ref{p9}) and (\ref{p10}), Eq. (\ref{p12}) becomes%
\[
a_{3}=\frac{L_{p,q,1}^{2}(x)}{2\left(  \mu+\lambda+2\xi\delta\right)  ^{2}%
}\left(  r_{1}^{2}+s_{1}^{2}\right)  +\frac{L_{p,q,1}}{2\left(  \mu
+2\lambda+2\xi\delta\right)  }\left(  r_{2}-s_{2}\right)  .
\]

Thus applying (\ref{w1}), we conclude that%
\[
\left\vert a_{3}\right\vert \leq\frac{p^{2}(x)}{\left(  \mu+\lambda+2\xi
\delta\right)  ^{2}}+\frac{\left\vert p(x)\right\vert }{\left(  \mu
+2\lambda+2\xi\delta\right)  }.
\]

\end{proof}

By setting $\mu=\delta=0$ and $\lambda=1$ in Theorem \ref{thm111}, we obtain
the following consequence.

\begin{corollary}
\bigskip If $f$ belongs to the class $\mathfrak{B}_{\Sigma}(1)\mathcal{=S}%
_{\Sigma}^{\ast}$ of bi-starlike functions$,$ then%
\[
\left\vert a_{2}\right\vert \leq\frac{2\left\vert p(x)\right\vert
\sqrt{\left\vert p(x)\right\vert }}{\sqrt{\left\vert  2p^{2}%
(x)-2\left(  p^{2}(x)+2q(x)\right)  \right\vert }},
\]
and%
\[
\left\vert a_{3}\right\vert \leq p^{2}(x)+\left\vert p(x)\right\vert .
\]

\end{corollary}

\section{\bigskip Fekete--Szeg\"{o} problem for the function class
$\mathfrak{B}_{\Sigma}^{\mu}(\lambda,\delta)$}

In this section, we aim to provide Fekete--Szeg\"{o} inequalities for
functions in the class $\mathfrak{B}_{\Sigma}^{\mu}(\lambda,\delta)$. These
inequalities are given in the following theorem.

\begin{theorem}
\label{thm1} For $\lambda\geq1,$ $\mu\geq0$ and $\delta\geq0$, let
$f\in\mathcal{A}$ belongs to the class $\mathfrak{B}_{\Sigma}^{\mu}%
(\lambda,\delta).$ Then

$\hspace{1in} \left\vert a_{3}-\upsilon a_{2}^{2}\right\vert \leq \left\{
\begin{array}
[c]{c}%
\frac{\left\vert p(x)\right\vert }{\left(  \mu+2\lambda+2\xi\delta\right)  },\\ \\
\frac{2\left\vert p(x)\right\vert ^{3}\left\vert 1-\upsilon\right\vert
}{\left\vert p(x) \Upsilon(x) \right\vert },%
\end{array}
\right.
\begin{array}
[c]{c}%
\left\vert \upsilon-1\right\vert \leq\frac{1}{2\left(  \mu+2\lambda+2\xi
\delta\right)  }\times\left\vert
\Upsilon(x)
\right\vert \\ \\
\left\vert \upsilon-1\right\vert \geq\frac{1}{2\left(  \mu+2\lambda+2\xi
\delta\right)  }\times\left\vert
\Upsilon(x)
\right\vert ,
\end{array}
$

\hspace{-.25in} where $\Upsilon(x)=\left(  \mu+2\lambda\right)  \left[  1+\mu+\frac{12\delta}{2\lambda+1}\right]
p(x)-2\left(  \mu+\lambda+2\xi\delta\right)  ^{2}\frac{p^{2}(x)+2q(x)}{p(x)}$.
\end{theorem}

\begin{proof}
From (\ref{p11}) and (\ref{p12})%
\begin{align*}
a_{3}-\upsilon a_{2}^{2}  &  =\left(  1-\upsilon\right)  \frac{L_{p,q,1}%
^{3}(x)\left(  r_{2}+s_{2}\right)  }{\left[  \left(  \mu+2\lambda\right)
\left[  1+\mu+\frac{12\delta}{2\lambda+1}\right]  L_{p,q,1}^{2}(x)-2\left(
\mu+\lambda+2\xi\delta\right)  ^{2}L_{p,q,2}(x)\right]  }\\
&  +\frac{L_{p,q,1}}{2\left(  \mu+2\lambda+2\xi\delta\right)  }\left(
r_{2}-s_{2}\right) \\
&  =L_{p,q,1}\left[  \left[  \varphi(\upsilon,x)+\frac{1}{2\left(
\mu+2\lambda+2\xi\delta\right)  }\right]  r_{2}+\left[  \varphi(\upsilon
,x)-\frac{1}{2\left(  \mu+2\lambda+2\xi\delta\right)  }\right]  s_{2}\right]
,
\end{align*}
where%
\[
\varphi(\upsilon,x)=\frac{L_{p,q,1}^{2}(x)\left(  1-\upsilon\right)  }{\left[
\left(  \mu+2\lambda\right)  \left[  1+\mu+\frac{12\delta}{2\lambda+1}\right]
L_{p,q,1}^{2}(x)-2\left(  \mu+\lambda+2\xi\delta\right)  ^{2}L_{p,q,2}%
(x)\right]  },
\]

Then, in view of (\ref{w1}), we conclude that%
\[
\left\vert a_{3}-\upsilon a_{2}^{2}\right\vert \leq\left\{
\begin{array}
[c]{c}%
\frac{\left\vert p(x)\right\vert }{\left(  \mu+2\lambda+2\xi\delta\right)  }\\
2\left\vert p(x)\right\vert \left\vert \varphi(\upsilon,x)\right\vert
\end{array}
\right.
\begin{array}
[c]{c}%
0\leq\left\vert \varphi(\upsilon,x)\right\vert \leq\frac{1}{2\left(
\mu+2\lambda+2\xi\delta\right)  },\\
\left\vert \varphi(\upsilon,x)\right\vert \geq\frac{1}{2\left(  \mu
+2\lambda+2\xi\delta\right)  }.
\end{array}
\]

Which completes the proof of Theorem \ref{thm1}.
\end{proof}

Putting $\mu=\delta=0$ and $\lambda=1$ in Theorem \ref{thm1}, we conclude the
following result:

\begin{corollary}
\bigskip If $f$ belongs to the class $\mathcal{S}_{\Sigma}^{\ast},$ then%
\[
\left\vert a_{3}-\upsilon a_{2}^{2}\right\vert \leq\left\{
\begin{array}
[c]{c}%
\left\vert p(x)\right\vert, \\ \\
\frac{2\left\vert p(x)\right\vert ^{3}\left\vert 1-\upsilon\right\vert
}{4\left\vert q(x)\right\vert },%
\end{array}
\right.
\begin{array}
[c]{c}%
\left\vert \upsilon-1\right\vert \leq\left\vert \frac{q(x)}{p(x)}\right\vert
\\ \\
\left\vert \upsilon-1\right\vert \geq\left\vert \frac{q(x)}{p(x)}\right\vert
\end{array}
.
\]

\end{corollary}

Putting $\upsilon=1$ in Theorem \ref{thm1}, we conclude the following result:

\begin{corollary}
\bigskip If $f$ belongs to the class $\mathcal{S}_{\Sigma}^{\ast},$ then%
\[
\left\vert a_{3}- a_{2}^{2}\right\vert \leq\frac{\left\vert
p(x)\right\vert }{\left(  \mu+2\lambda+2\xi\delta\right)  }.
\]

\end{corollary}


\end{document}